\documentclass{amsart}
\usepackage{amsmath,amsfonts,amsthm,amssymb}
\usepackage[alphabetic]{amsrefs}
\usepackage{graphicx}
\usepackage[all]{xypic}

\newtheorem{theorem}{Theorem}[section]
\newtheorem{proposition}[theorem]{Proposition}
\newtheorem{corollary}[theorem]{Corollary}
\newtheorem{lemma}[theorem]{Lemma}

\newtheorem*{up}{Universal Property of $\mathcal{L}(X)$}

\theoremstyle{definition}

\newtheorem{remark}[theorem]{Remark}
\newtheorem{definition}[theorem]{Definition}

\newtheorem{example}[theorem]{Example}

\long\def\symbolfootnote[#1]#2{\begingroup%
\def\thefootnote{\fnsymbol{footnote}}\footnote[#1]{#2}\endgroup}

\newcommand{\field}[1]{\ensuremath{\mathbb{#1}}}

\newcommand{\inv}{^{-1}}
\newcommand{\D}{\mathrm{d}}
\newcommand{\id}{\mathrm{id}}

\newcommand{\Aff}{\mathrm{Aff}}

\newcommand{\ttil}{\widetilde}
\newcommand{\N}{\field{N}}
\newcommand{\R}{\field{R}}
\newcommand{\C}{\field{C}}
\newcommand{\la}{\mathcal{L}}
\newcommand{\Q}{\field{Q}}

\newcommand{\Z}{\field{Z}}

\newcommand{\lx}{\mathcal{L}(x)}
\newcommand{\lX}{\mathcal{L}(X)}
\newcommand{\leX}{\mathcal{L}^{\varepsilon}(X)}
\newcommand{\leXp}{\mathcal{L}^{\varepsilon'}(X)}

\renewcommand{\phi}{\varphi}

\newcommand{\eps}{\varepsilon}
\newcommand{\bdy}{\partial}                       

\newcommand{\Lie}[1]{\mathrm{#1}}

\newcommand{\bb}{\tilde{B}(x,r)^t}

\begin{document}

\begin{center}
{\large\bf Wild singularities of flat surfaces} 
\smallskip

{\bf
Joshua P.~Bowman$^a$ 
\& 
Ferr\'an Valdez$^b$ 
} 
\medskip

$^a$ IMS, Stony Brook University, \\
Stony Brook, NY, USA \\
\emph{e-mail:} \texttt{joshua.bowman@gmail.com}
\smallskip

$^b$ Centro de Ciencias Matem\'aticas, U.N.A.M.\\                   
Campus Morelia, Michoac\'an, M\'exico\\
\emph{e-mail:} \texttt{ferran@matmor.unam.mx}
\end{center}

\title{}

\begin{abstract}
\noindent
We consider flat surfaces and the points of their metric completions, 
particularly the singularities to which the flat structure of the 
surface does not extend. The local behavior near a singular point $x$ 
can be partially described by a topological space $\mathcal{L}(x)$ 
which captures all the ways that $x$ can be ``approached linearly''. 
The homeomorphism type of $\mathcal{L}(x)$ is an affine invariant. 
When $x$ is not a cone point or an infinite-angle singularity, we say 
it is {\em wild}; in this case it is necessary to add further metric 
data to $\mathcal{L}(x)$ to get a quantitative description of the 
surface near $x$.
\end{abstract}
\maketitle




The study of flat surfaces, appearing under different guises 
(quadratic differentials, abelian differentials, translation surfaces, 
measured foliations, F-structures, and so on), reaches back at least 
to the 1970--80s, when seminal work of Thurston, Masur, Veech, and 
others uncovered fundamental connections among surface automorphisms, 
flat surface geometry, and billiard dynamics. However, their origins 
go back much further to the 1930--40s, with Nielsen's classification 
of torus automorphisms and Fox--Kerschner's association of a Riemann 
surface to billiards in a polygon \cite{FoxKersh}, sometimes called the 
Katok--Zemlyakov unfolding construction. Throughout much of the 
history of flat surfaces, the focus has been on {\em compact} flat 
surfaces, having so-called ``cone-type'' singularities, with 
non-compact surfaces appearing only sporadically. In this way, 
researchers could bring to bear the considerable power of 
finite-dimensionality in Teichm\"uller theory and in algebraic 
constructions such as homology groups.

In recent years, increasing attention has been paid to the study of 
non-compact flat surfaces, or more precisely surfaces of infinite type. 
Several treatments deal with classes of examples such as covers of 
compact surfaces \cite{HWS} or surfaces arising from certain dynamical systems 
(wind-tree models \cite{HLT}, irrational billiards \cite{V}, exchanges of infinitely many 
intervals \cite{Hoo}, etc.). In a similar vein, de~Carvalho--Hall have initiated 
a study of dynamical systems on genus-zero surfaces with infinitely 
many singularities \cite{dCH}. These studies have necessitated the adaptation of 
tools from the theory of compact flat surfaces, but have so far 
remained fuzzy on the local, intrinsic behavior of a surface near its 
singular points. The simple description via cone points becomes 
inadequate when the metric structure imposed on a surface can allow 
for essentially arbitrary topological complication within a bounded 
region. In our opinion, this constitutes an important lack and an 
obstacle to properly understanding basic notions such as straight-line 
flow and deformations of flat surfaces.

Here we present a method for studying the local behavior of 
singularities of topologically infinite flat surfaces. For the most 
part, we restrict our attention to {\em isolated} singularities in 
the metric completion of a flat surface. These singularities do not, 
in general, have an analytic description parallel to the description 
of cone points as zeroes of holomorphic differentials. Nor are they 
determined (up to local isometry) by discrete sets of data. 

Our primary invariant is a topological space $\mathcal{L}(x)$ 
associated to each point $x$ in the metric completion of a flat 
surface, which may be thought of as a set of directions arising from $x$; each 
element of $\mathcal{L}(x)$ represents a ``linear approach'' to $x$. 
This space is invariant under affine deformations of the surface. 
We then add metric data to $\mathcal{L}(x)$ that quantify how the 
linear approaches are distributed. Together, these provide a 
complete description of the surface near $x$.

From our perspective, this paper provides a unifying vision of 
the disparate singular behaviors that had previously been only 
superficially observed. In \S\ref{S:definitions} we recall some 
motivating examples and define the invariant $\mathcal{L}(x)$ along 
with its global version $\mathcal{L}(X)$, where $X$ is a flat surface 
and $x$ is in the metric completion of $X$. In \S\ref{S:topology} we 
examine the topological structure of $\mathcal{L}(X)$ and 
$\mathcal{L}(x)$ in more detail. We prove in particular that the space
$\mathcal{L}(X)$ is an extension of the unit tangent bundle of the flat
surface $X$ to its metric completion.  In \S\ref{S:affine} we study the 
effect of affine maps on $\mathcal{L}(X)$ and $\mathcal{L}(x)$. 
In \S\ref{S:isometries} we provide a set of necessary and sufficient 
conditions for two points to have isometric neighborhoods. Finally, 
in \S\ref{S:remarks} we briefly compare our invariants with similar 
constructions that have been previously described. Our constructions 
are quite general and would apply in many contexts outside of flat 
surfaces, while they also retain extra available information due to 
properties of flat surfaces that set them apart from general metric 
spaces.

\subsection*{Acknowledgements}

We thank CIRM in Luminy and the Hausdorff Institute in Bonn for 
hosting two conferences where many of these ideas were formulated; 
we are grateful to the organizers of these conferences, as well. 
We thank our home universities for facilitating remote discussion. 
The second author thanks financial support from CONACYT, IACOD and
PAPIIT. We thank our colleagues for helpful and encouraging conversations, 
especially Barak Weiss, Pat Hooper, Gabriela Schmith\"usen, and 
John Smillie.

\section{Basic definitions and examples.}\label{S:definitions}

In this section we introduce the basic definitions and examples. The main objects we will be working with are translation surfaces arising from holomorphic 1-forms on a fixed Riemann surface.

\begin{definition}\label{flat surface}
A \emph{flat surface} is a pair $(X,\omega)$ formed by a Riemann 
surface $X$ and a non identically zero holomorphic $1$-form $\omega$ 
on $X$. Where clarity permits, we abbreviate $(X,\omega)$ by $X$.
We denote by $Z(\omega)\subset X$ the set of zeroes of $\omega$.
\end{definition}

Sometimes we will also use the terminology \emph{translation surface} 
to refer to a flat surface. Local integration of the form $\omega$
endows $X':=X\setminus  Z(\omega)$ with an atlas whose transition 
functions are translations of $\C$. The pullback of the standard 
translation invariant flat metric on the complex plane defines a flat 
metric $d_X$ on $X\setminus Z(\omega)$. We will denote by 
$\widehat{X}$ the metric completion of $X'$ with respect to $d_X$. 
In this article we will work with flat surfaces satisfying the 
following:
\medskip

\noindent
\underline{Main hypothesis.} 
\emph{The set $\mathrm{Sing}(X):=\widehat{X}\setminus X'$ 
is a discrete subset of $\widehat{X}$}.
\medskip

Remark that $\widehat{X}$, and hence $\mathrm{Sing}(X)$, 
depends on our choice of the $1$-form $\omega$ on $X$. 
Points in $\mathrm{Sing}(X)$ fall in one of the following cases.
\begin{enumerate}
\item Flat points. These are points $p\in\widehat{X} \setminus X$ 
for which the flat metric of $X$ extends to a flat metric on 
$X\cup\{p\}$.
\item Finite angle singularities. These are points $p\in\widehat{X}$ 
for which the Riemann surface structure of $X$ extends to 
$X\cup\{p\}$. In a neighborhood of $p$ the form $\omega$ is given 
by $z^k\,\D{z}$ for some $k\in\N$.
\item Infinite angle singularities. For each of these singularities 
$p\in\widehat{X}$, there exists a punctured neighborhood 
$0<d_X(w,p)<\eps$ which is isometric to an infinite cyclic covering 
of the punctured disc $(0<|z|<\eps, \D{z})$. Such punctured 
neighborhoods can be pictured as an infinite double helicoid whose 
axis has been collapsed to a point. Infinite angle singularities 
naturally appear in flat surfaces asociated to irrational polygonal 
billiards. 
\item The rest. We call such points $p$ \emph{wild} singularities 
of the flat surface. These points and their neighborhoods 
$0<d_X(w,p)<\eps$ will constitute the main research point of 
this article. 
\end{enumerate}

\noindent
\emph{Convention.} Henceforth we will work only with flat surfaces 
$(X,\omega)$ such that the set of flat points in $\mathrm{Sing}(X)$ 
is empty.

\begin{definition}[Saddle connection]
A {\em critical trajectory} of a flat surface $(X,\omega)$ is an 
open geodesic in the flat metric $d_X$ whose image under the natural 
embedding $X\hookrightarrow\widehat{X}$ issues from a point in 
$\mathrm{Sing}(X)$, contains no other point of $\mathrm{Sing}(X)$ 
in its interior and is not property contained in some other geodesic 
segment. A \emph{saddle connection} is a finite length critical 
trajectory.
\end{definition}

\begin{definition}[Veech group]
Let $\Aff_+(X,\omega)$ be the group of affine orientation preserving 
homeomorphisms of $(X,\omega)$. Consider the map that associates to 
each $\phi\in\Aff_+(X,\omega)$ its Jacobian derivative 
$D\phi\in\Lie{GL}_+(2,\R)$. We call the image of this map the 
\emph{Veech group} of $(X,\omega)$ and denote it by $\Gamma(X)$.
\end{definition}

In the following paragraphs we introduce the topological spaces $\lX$ and $\lx$. They will play the role of unit tangent bundle and unit tangent space on $\widehat{X}$ respectively.
 
\begin{definition}[Linear approach] 
\label{linap}
Given $\eps>0$, let $\mathcal{L}^\eps(X)$ be the space 
\[
\mathcal{L}^{\eps}(X)
:=\{\text{unit speed geodesic trajectories}\ \gamma:(0,\eps)\to X'\}.
\]
Each $\leX$ carries the uniform topology, defined by the uniform 
metric:
\[
d_{\eps}(\gamma_1,\gamma_2)
= \sup_{0<t<\eps}d_X\big(\gamma_1(t),\gamma_2(t)\big)
\]
Two elements $\gamma_1\in\leX$ and $\gamma_2\in\leXp$ are said to be 
equivalent if and only if $\gamma_1(t)=\gamma_2(t)$ for all 
$t\in(0,\min\{\eps,\eps'\})$. We denote by $\sim$ this equivalence 
relation and define:
\begin{equation}
\lX:=\bigsqcup_{\eps>0}\leX{\huge /}\sim
\end{equation}
The equivalence class of $\gamma$ will be denoted by $[\gamma]$. 
We call each element $[\gamma]$ of $\mathcal{L}(X)$ a 
\emph{linear approach} to the point 
$\lim_{t\to 0}\gamma(t)\in\widehat{X}$.
\end{definition}

\subsection*{Topology for $\lX$}
For each $\eps'\leq\eps$ the restriction of each linear approach in 
$\leX$ to the interval $(0,\eps')$ defines a continuous injection:
\[
\rho_{\eps}^{\eps'}:\leX\to\leXp
\]  
Define $\eps\trianglelefteq\eps'$ if and only if $\eps'\leq\eps$, 
where $\leq$ is the standard order in $\R$. Then 
$\langle\leX,\rho_{\eps}^{\eps'}\rangle$ is a direct system of 
topological spaces over $(\R^+,\trianglelefteq)$. Since for every 
$\eps\trianglelefteq\eps'\trianglelefteq 0$ the projection map 
$\gamma\mapsto [\gamma]$ from $\leX$ to $\lX$ is injective and 
commutes with $\rho_{\eps}^{\eps'}$, we have the equality of sets
\begin{equation}\label{lX}
\lX=\varinjlim \leX
\end{equation}
Henceforth we endow $\lX$ with the {\em direct limit topology} 
(sometimes called the {\em final topology}), which is the finest 
topology such that the inclusions $\mathcal{L}^\eps(X) \to \lX$ are 
all continuous. We denote by $\rho_{\eps}:\leX\to\lX$ the natural 
projection $\gamma \mapsto [\gamma]$. Unless otherwise stated, we also 
identify $\mathcal{L}^\eps(X)$ with its image in $\mathcal{L}(X)$.

\begin{definition}
Let $x\in\widehat{X}$. We define $\lx$ to be the set of all linear 
approaches $[\gamma]\in\lX$ such that $\lim_{t\to 0}\gamma(t)=x$, 
endowed with the subspace topology.
\end{definition}

The space $\lx$ naturally decomposes into rotational components, 
which we define in the following paragraphs.

\begin{definition}[Angular sector]\label{angsec}
We call an \emph{angular sector} a triple of the form $(I,c,i_c)$, 
where $I\subseteq\R$ is a non-empty \emph{interval}, $c\in\R$ is a 
constant and $i_c$ is a isometry into $X'$ of the open set
\begin{equation}\label{U}
U=U(I,c):=\{(x,y)\mid x<c,\, y\in I\},
\end{equation}
endowed with the translation structure defined by the holomorphic 
$1$-form $e^z\,\D{z}$ (where $z=x+iy$).  
\end{definition} 

Observe that for every fixed angular sector $(I,c,i_c)$ the limit 
$\lim_{x\to-\infty}i_c(x,y)$ exists in $\widehat{X}$ and is 
independent from the $y$-coordinate in $U(I,c)$ into $X$.
\medskip

\noindent
\emph{Convention}: All sets $U$ are contained in the same copy of 
$\R^2$ on which we have previously fixed our favorite orientation. 
The interval $I$ in the preceding definition can be just a point, 
unbounded and are not necessarily closed or open.

\begin{definition}[Rotational component]
Let $[\gamma_1]$ and $[\gamma_2]$ be two linear approaches in 
$\la(x)$. We say that $[\gamma_1]$ and $[\gamma_2]$ are equivalent 
if and only if there exist representatives $\gamma_i:(0,\eps_i)\to X$, 
$i=1,2$ and an angular sector $(I,c,i_c)$ such that 
$(i_c^{-1}\circ \gamma_i)(0,\eps_i)$ is equal to an infinite segment 
of real line $(x<c,y_i)$, for some fixed $y_i\in I$, $i=1,2$.
We denote by $\overline{[\gamma]}$ the equivalence class defined by 
$[\gamma]\in \lx$, and we call this class \emph{the rotational 
component of $\mathcal{L}(x)$ containing $[\gamma]$}. 
 \end{definition}

\begin{lemma}
Every rotational component $\overline{[\gamma]}$ containing more than 
one element admits a connected real 1-manifold translation structure, 
possibly with non-empty boundary.
\end{lemma}
\begin{proof}
For every angular sector $(I,c,i_c)$ making two linear approaches 
in $\overline{[\gamma]}$ equivalent, we define 
\[
V = V(I,c,i_c) = \{i_c(x,y) \mid x<c \hspace{1mm}\}_{y\in I} 
\quad\subseteq\quad\overline{[\gamma]}.
\]
Call $\mathcal{V}$ the collection of all the $V(I,c,i_c)$ 
obtained by considering angular sectors $(I,c,i_c)$ as before. This 
collection is the basis for a topology on $\overline{[\gamma]}$. 
With respect to this topology the class $\overline{[\gamma]}$ is 
Hausdorff, second countable and connected. Define $\phi_V:V\to I$ 
by $\phi_V [i_c(x<c,y)]= y$. This is a local homeomorphism. Given 
the convention made after definition \ref{angsec}, the set 
$\{(V,\phi_V)\}_{V\in\mathcal{V}}$ defines an atlas on 
$\overline{[\gamma]}$ whose transition functions are translations 
in $\R$. Remark that charts for boundary points are defined by left 
or right closed intervals.
\end{proof}

\begin{remark} \label{R:rotcom} From now on $\overline{[\gamma]}$ 
will denote the rotational component defined by the linear approach 
$[\gamma]$ and endowed with the translation structure given by the 
preceding proposition. We can lift the standard translation invariant 
metric of $\R$ to each $\overline{[\gamma]}$. We have the following 
situations:
\begin{enumerate}
\item If $\lx$ contains a compact rotational component 
$\overline{[\gamma]}$, there are two possibilities:
\begin{enumerate}
\item[(1.a)] The rotational component is an interval $[a,b]$, perhaps 
with $a=b$. In this case the rotational component is a proper subset 
of $\lx$.
\item [(1.b)]The rotational component is homeomorphic to $S^1$. 
In this case $\lx$ and $\overline{[\gamma]}$ are homeomorphic as 
topological spaces.
\end{enumerate}

\item If $\lx$ contains a non compact rotational component 
$\overline{[\gamma]}$. The following situations can occur:
\begin{enumerate}
\item[(2.a)] The total length of $\overline{[\gamma]}$ is finite. 
In this case the rotational component is isometric to a bounded 
interval and is a proper subset of $\lx$.
\item[(2.b)] The total length of $\overline{[\gamma]}$ is infinite, 
but the class is isometric to an unbounded proper interval of $\R$. 
In this case we say that the class $\overline{[\gamma]}$ is a 
\emph{spire}. A spire may coincide with or be a proper subset of 
$\lx$.
\item[(2.c)]  The total length of $\overline{[\gamma]}$ is infinite, 
and the class is isometric to $\R$. In this case we say that the 
class $\overline{[\gamma]}$ is a \emph{double spire}. This case 
contains all infinite angle singularities, for which necessarily 
$\lx=\overline{[\gamma]}$. Nevertheless, there are examples for which 
$\lx$ is a double spire \emph{but} $x$ \emph{is not} an infinite 
angle singularity.
\end{enumerate}

\end{enumerate}
\end{remark}

\begin{remark}
Let $x\in\widehat{X}$ be a flat point, a finite angle singularity of 
angle $2\pi k$, $k>1$, or an infinite angle singularity. Then $\lx$ 
is isometric to $\R/2\pi\Z$, $\R/2\pi k\Z$ or $\R$ respectively.
\end{remark}

\subsection{Examples}\label{SS:examples}

In the rest of this section we present some examples of translation surfaces having wild singularities.

\begin{example}[\cite{Ch,CGL}]\label{Ex:chamanara} 
This example appears naturally when studying the self mappings of the unit square known as the \emph{horseshoe} and \emph{baker's} map. Start with a unit square $S$, let $\alpha=1/2$ and partition its top edge into segments of lengths $\alpha^k$, $1 \le k < \infty$, in decreasing order from left to right. Do the same for the bottom edge, but in reverse order. Remove extremities of all segments involved and identify (open) segments of the same length via translation. Now partition the left edge from top to bottom in the same way, and the right edge from bottom to top, remove extremities of all segments involved and identify those of the same length via translation. (See Figure~\ref{F:chamanara-ex}.) The result is an open flat Riemann surface $X_{\alpha}$ of infinite genus with one end (\emph{a.k.a.}\ a Loch Ness monster after Ghys \cite{Ghys}). The metric completion 
$\overline{X_{\alpha}}$ is obtained by adding the extremities of each $A_i$ and $B_i$, $i\in\N$. In $\overline{X_{\alpha}}$ all this added points are at distance zero from each other, hence 
$\overline{X_{\alpha}}\setminus X_{\alpha}={x}$. The horizontal and vertical flow on $X_{\alpha}$ define two families of saddle connections whose length is not bounded away from zero. Therefore $x$ is a wild singularity. Chamanara observes that ``[g]eometrically, the surface spirals infinitely many times around this point.'' In fact, 
$\mathcal{L}(x)$ decomposes into an infinite number of rotational components.
\begin{figure}
\centering
\includegraphics[scale=0.8]{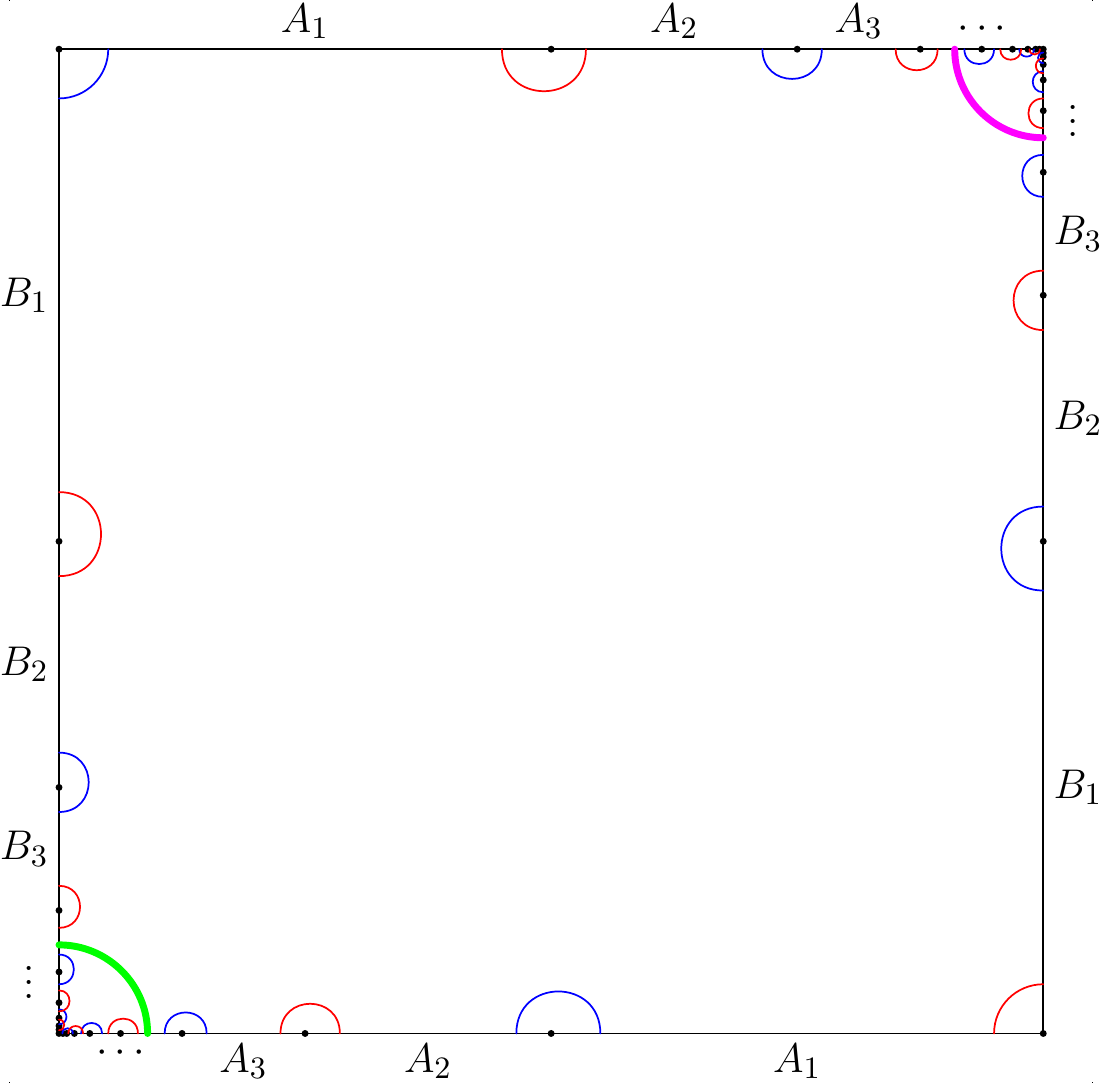}
\begin{center}
\caption{Double spires and finite length rotational components in  $X_{1/2}$.}\label{F:chamanara-ex}
\end{center}
\end{figure}
Indeed, if the intersection of the diagonals in the unit square is 
the origin and points in $A_1\cap B_1$ are $(-\frac{1}{2},\frac{1}{2})$ and $(\frac{1}{2},-\frac{1}{2})$ , 
then $\gamma_1(t):=(1-t)( -\frac{1}{2},\frac{1}{2})$ and $\gamma_2(t):=-\gamma_1(t)$, define two 
rotational components which are double spires. On the other hand, 
$\eta_1(t)=(1-t)(\frac{1}{2},\frac{1}{2})$ and $\eta_2(t)=-\eta_1(t)$, define two rotational 
components whose length is $\pi/4$. Remark that the obstructions 
for these rotational components to become spires are precisely the 
horizontal and vertical saddle connections defined by $A_i$ and $B_i$, 
$i\in\N$. Using Chamanara's results \cite[Theorem B]{Ch} about the Veech
group of $X_{\alpha}$ one can prove that $\mathcal{L}(x)$ decomposes into countably 
many rotational components. The preceding argumentation remains valid 
if we change $\alpha=1/2$ for $\alpha=1/n$ with $n\in\N$. 
\end{example}

\begin{example}[The geometric series construction]\label{Ex:geoconst}
We introduce a local construction that will provide an archetype for half spires. Let $0<\alpha<1$, $I_0=[0,\alpha]$ and, for each $n\geq 1$ define $I_n=[\sum_{m=1}^n\alpha^m,\sum_{m=1}^{n+1}\alpha^m]$. Let $J_0=[1-\alpha,1]$ and $J_n=[1-\sum_{m=1}^{n+1}\alpha^m,1-\sum_{m=1}^n\alpha^m]$, $n\geq 1$. Now slit the xy-plane along $[0,1]$. Points in the boundary of the slitted plane are thought as two different copies of $[0,1]$, each of which we identify  with the families of segments $\cup_{n\geq 0} I_n$ and $\cup_{n\geq 0} J_n$ respectively. Using translations, we identify for each $n\geq 0$ the segment $I_n$ with $J_n$ and remove the extremities of the segments involved at each step. The result is a flat surface $Y_{\alpha}$ such that $\mathrm{Sing}(\widehat{Y_{\alpha}})=x$. The space $\lx$ is formed by two rotational components whose representatives are the linear approaches $\gamma_1(t)=(0,t)$ and $\gamma_2(t)=(1,t)$, $t\in(0,1)$. Each the rotational component is isometric to $(0,\infty)$.
\end{example}

\begin{example}[Double parabola]\label{Ex:doubleparabola}
In this example we construct rotational components consisting of only one point. Let $\pm I_n$ be a family of segments in the $xy$-plane whose endpoints are $(\pm 2^n,2^{2n})$ and $(\pm 2^{n+1},2^{2(n+1)})$, 
$n\in\Z$. Let $\pm J_n$ be the family of segments whose endpoints are 
$(\pm 2^n,-2^{2n})$ and $(\pm 2^{n+1},-2^{2(n+1)})$, $n\in\Z$. Let $P_-$ be closure of the connected component of 
$\R^2\setminus\{\pm I_n\}_{n\in\Z}\cup(0,0)$ containing the negative $x$-axis. Analogously, let $P_+$ be the closure of the connected component of 
$\R^2\setminus\{\pm J_n\}_{n\in\Z}\cup(0,0)$ containing the positive $x$-axis. By construction $\partial P_-=\{-I_n\}_{n\in\Z}\cup(0,0)\cup \{-J_n\}_{n\in\Z}$ and 
$\partial P_+=\{I_n\}_{n\in\Z}\cup(0,0)\cup \{J_n\}_{n\in\Z}$. Remove all vertices 
(and the origin) from $P_-$ and $P_+$ and identify this two disjoint domains along parallel sides of the same length using translations. The result of this construction is a flat surface $X$ for which $\mathrm{Sing}(\widehat{X})$ is only one wild singularity $x$. The rotational components defined by 
$\pm\gamma(t)=(\pm t,0)$ consist of only one point. It is easy to check that in this case $\mathcal{L}(x)$ contains also two double spires.
\end{example}

Other examples of translation surfaces with wild singularities 
include Hooper's generalization of Thurston's construction to 
infinite bipartite graphs \cite{Hoo} or the geometric limit of the 
Arnoux--Yoccoz family studied by the first author \cite{B}.

\section{Topology of $\lX$ and $\lx$}\label{S:topology}

\subsection{Universal property of the direct limit}

In the previous section, we defined $\lX$ as the direct limit of 
the topological spaces $\mathcal{L}^\eps(X)$ and thereby endowed 
$\mathcal{L}(x)$ with the direct limit topology. In this section 
and the next we will explore some of the consequences of this 
topology. First, we state the corresponding universal property of 
$\lX$. Recall that for $\eps' < \eps$, $\rho_{\eps}^{\eps'}$ is the 
restriction map $\mathcal{L}^\eps(X) \to \mathcal{L}^{\eps'}(X)$, 
and for all $\eps > 0$, $\rho_{\eps}$ is the map 
$\gamma \mapsto [\gamma]$.

\begin{up}
Let $X$ be a translation surface. Given any topological space 
$\mathfrak{T}$ and any collection of continuous maps 
$\{f_\eps : \mathcal{L}^\eps(X) \to \mathfrak{T}\}_{\eps>0}$ 
such that $f_{\eps'} \circ \rho_{\eps}^{\eps'} = f_\eps$ 
whenever $\eps' < \eps$, there is a unique continuous map 
$f : \lX \to \mathfrak{T}$ such that $f \circ \rho_{\eps} 
= f_\eps$ for all $\eps$.
\end{up}

\subsection{Continuity of the basepoint and direction maps}

Two functions we would like to define on $\lX$ are the 
{\em basepoint} and {\em direction} maps, respectively 
denoted $\mathrm{bp}$ and $\mathrm{dir}$, and given on 
$\mathcal{L}^\eps(X)$ by 
\begin{align*}
\mathrm{bp} :&\ 
\gamma \mapsto \lim_{t\to0} \gamma(t) \in \widehat{X} \\
\mathrm{dir} :&\ 
\gamma \mapsto \dot\gamma(t) \in S^1 
&& \text{for any $t \in (0,\eps)$}.
\end{align*}
The basepoint map satisfies 
$\mathrm{bp} \circ \rho_{\eps}^{\eps'} = \mathrm{bp}$ because 
it does not depend on the length of the domain of $\gamma$, 
only on its values near zero. The direction map is well-defined 
because the translation structure of $X$ yields a trivialization 
of the unit tangent bundle $T_1(X) = X \times S^1$. Since each
geodesic is contained in a fiber of the projection $X\times S^1\to S^1$
the direction map satisfies $\mathrm{dir} \circ \rho_{\eps}^{\eps'} = 
\mathrm{dir}$.
\begin{proposition}\label{P:bp&dir}
For any $\eps > 0$, the functions 
$\mathrm{bp} : \mathcal{L}^\eps(X) \to \widehat{X}$ and 
$\mathrm{dir} : \mathcal{L}^\eps(X) \to S^1$ are continuous.
\end{proposition}

\begin{proof}
Let $x \in \widehat{X}$ and $r > 0$. Suppose 
$d_X(x,\mathrm{bp}(\gamma)) < r$, and set 
$r' = r - d_X(x,\mathrm{bp}(\gamma))$. Then for any 
$\eta \in \mathcal{L}^\eps(X)$ such that $d_\eps(\gamma,\eta) 
< r'$, 
\[
d_X(x,\mathrm{bp}(\eta)) 
\le d_X(x,\mathrm{bp}(\gamma)) + d_\eps(\gamma,\eta) 
< d_X(x,\mathrm{bp}(\gamma)) + r' = r
\]
Thus the preimage of the open ball $B(x,r)$ is open in 
$\mathcal{L}^\eps(X)$. Therefore the basepoint map is continuous. 
To show that the direction map is continuous, fix $\gamma \in 
\mathcal{L}^\eps(X)$ and $\theta \in (0,\pi/2)$. Choose $a$ and 
$b$ such that $0 < a < b < \eps$. The segment $\gamma([a,b]) 
\subset X'$ is compact. Hence we can find $\delta>0$ such that
$U:=\bigcup_{t\in[a,b]}B(\gamma(t),\delta)$ is isometric to a 
Euclidean rectangle capped by two half discs of arbitrary small area. 
(One can think of $U$ as the barrel of a gun, through which we want 
to aim trajectories sufficiently close to $\gamma$.) We develop $U$ 
in the plane, where elementary geometry shows that $\delta$ can be 
chosen so that for every $\eta$ with  $d_{\eps}(\eta,\gamma)<\delta$ 
the direction $\mathrm{dir}(\eta)$ lies in a neighborhood centered at 
$\mathrm{dir}(\gamma)$ of radius $\theta$.
\end{proof}

Now the universal property of $\lX$ implies the following:

\begin{corollary}
The functions $\mathrm{bp} : \lX \to \widehat{X}$ and 
$\mathrm{dir} : \lX \to S^1$ are continuous.
\end{corollary}

\subsection{A generating set for the topology on $\lX$}

The definition of $\lX$ as a direct limit makes its topology 
somewhat obscure. In this section we introduce a generating set for 
this topology that will be useful, as we will see later, to prove 
topological statements about $\lX$.

For any $x \in X$ and $r > 0$, let $B(x,r)$ denote the open $d_X$-ball 
in $X$ centered at $x$ and having radius $r$. Then, for any $t > 0$, 
we set 
\[
\tilde{B}(x,r)^t 
= \big\{ [\gamma] \in \lX \mid \gamma(t)\in B(x,r) \big\}.
\]
Implicit in this definition is the assumption that, in order for 
$[\gamma] \in \tilde{B}(x,r)^t$, $[\gamma]$ must have a representative 
of length greater than $t$. The following technical claim is our main 
result for this section.

\begin{proposition}\label{P:BxIsBasis}
The collection of sets
\begin{equation}\label{Eq:subbasis}
\mathcal{B} :=
\big\{ \bb \mid x\in X',\ r > 0,\ t > 0 \big\}
\end{equation}
generates the limit topology in $\lX$.
\end{proposition}

First, we describe the restriction of this topology to each of the 
spaces $\mathcal{L}^\eps(X)$. Recall that we denote by $d_\eps$ the 
uniform metric on $\mathcal{L}^\eps(X)$.

\begin{lemma}\label{L:Beps}
Let $\tau_\eps$ denote the topology on $\mathcal{L}^\eps(X)$ induced 
by $d_\eps$. Then 
\[
\mathcal{B}^\eps := 
\big\{ \tilde{B}(x,r)^t \mid x \in X',\ 0 < r,\ 0 < t < \eps \big\}
\]
generates $\tau_\eps$.
\end{lemma}

\begin{proof}
This is a straightforward variant of the well-known fact that the 
uniform topology on a collection of maps from one metric space to 
another coincides with the compact-open topology.
\end{proof}

\begin{proof}[Proof of Proposition~\ref{P:BxIsBasis}] 
Clearly $\mathcal{B}$ defines a covering of $\lX$. We denote by 
$\tau'$ the topology generated by $\mathcal{B}$ and by $\tau$ the 
limit topology on $\lX$. The proof requires two steps.

\emph{Step 1:} $\tau'\subseteq\tau$. That is, every element of 
$\mathcal{B}$ is open in the direct limit topology. This means 
we need to show that the intersection of each $\tilde{B}(x,r)^t 
\in \mathcal{B}$ with $\mathcal{L}^\eps(X) \subset \mathcal{L}(X)$ is 
open for every $\eps > 0$. If $t < \eps$, then by Lemma~\ref{L:Beps} 
we are done. If, on the other hand, $t \ge \eps$, then given 
$[\gamma] \in \tilde{B}(x,r)^t$ we need to find an open set in 
$\tilde{B}(x,r)^t \cap \mathcal{L}^\eps(X)$ containing $[\gamma]$. 
Choose $a$ such that $0 < a < \eps$ and take $\delta < r$ such 
that, as in the proof of Proposition~\ref{P:bp&dir}, the 
$\delta$-neighborhood $U$ of $\gamma([a,t])$ is isometric to a 
Euclidean rectangle capped by half-discs. Choose times 
$t_1, \dots, t_n$ and radii $r_1, \dots, r_n$ so that each  
$B(\gamma(t_n),r_n)$ lies in $U$ and $\{B(\gamma(t_n),r_n)\}$ covers 
$\gamma([a,\eps])$. Then every trajectory in 
$\bigcap_{i=1}^n \tilde{B}(\gamma(t_n),r_n)^{t_n}$ remains in $U$ 
from time $a$ to $\eps$. Let $\alpha \subset S^1$ be the arc around 
$\mathrm{dir}(\gamma)$ of length $\arctan(\delta/t)$. Then, by 
continuity of the direction map, $\mathrm{dir}\inv(\alpha)$ is open 
in $\mathcal{L}^\eps(X)$. Hence 
\[
\bigg( \bigcap_{i=1}^n \tilde{B}(\gamma(t_n),r_n)^{t_n} \bigg)
\cap \mathrm{dir}\inv(\alpha) 
\]
is an open set in $\tilde{B}(x,r)^t \cap \mathcal{L}^\eps(X)$ 
containing $[\gamma]$.

\emph{Step 2:} $\tau\subseteq\tau'$. Let $V\in\tau$. We need to show 
that each element of $V$ is contained in a finite intersection 
$\bigcap U_i$ of sets $U_i$ in $\mathcal{B}$, such that this 
intersection is itself contained in $V$. Choose $[\gamma]\in V$ 
having a representative $\gamma\in\leX$. Identify $\rho_\eps^{-1}(V)$ 
with $V_\eps:=V\cap \leX\in\tau_\eps$. Using the generating set for 
$\tau_\eps$ provided by Lemma~\ref{L:Beps}, we can find a finite 
collection $\{\tilde{B}_\eps(x_i,r_i)^{t_i}\}_{i=1}^n$ such that 
the intersection $I$ of its elements satisfies 
$[\gamma]\in I\subset V\cap \leX\subset V$. Hence $\tau\subset\tau'$.
\end{proof}

\subsection{Topological consequences}

In this section we explore some topological properties of $\lX$. 
Recall that we identify $\leX$ with its image in $\lX$ by $\rho_\eps$.

\begin{corollary}
For all $t > 0$, $\bigcup_{\eps>t} \leX$ is open in $\lX$.
\end{corollary}
\begin{proof}
If $[\gamma]$ has a representative $\gamma$ of length $\eps > t$ 
and $t'$ satisfies $t<t'<\eps$, then there exists some $r>0$ such 
that $\tilde{B}(\gamma(t'),r)^{t'}$ lies in $\leX\subset\lX$, 
contains $[\gamma]$ and consists of germs of trajectories $[\eta]$ 
having representatives of length greater than $t$ and satisfying 
the condition $d_X(x,\eta(t)) < r$.
\end{proof}

\begin{definition}[Maximal length function]\label{D:MLF}
Let $\ell$ be defined on $\lX$ by 
\[
\ell[\gamma] 
= \sup \left\{ \eps > 0 \mid 
               \eta \in \mathcal{L}^\eps(X),\
               \rho_\eps(\eta) = [\gamma]
       \right\}.
\]
This is the {\em maximal length function}; it measures the longest 
a representative of the class $[\gamma]$ can be. It takes values in 
the positive extended reals.
\end{definition}

\begin{corollary}
The function $\ell : \lX \to (0,\infty]$ is lower semi-continuous.
\end{corollary}

We recall that lower semi-continuity of a real-valued function 
$\phi$ on a topological space $\mathfrak{T}$ means either of the 
following (equivalent) properties holds:
\begin{itemize}
\item For all $t \in \R$, $\phi\inv((t,\infty))$ is open in 
$\mathfrak{T}$.
\item If $\{x_n\}_{n=1}^\infty$ is a sequence of points in 
$\mathfrak{T}$ such that $x_n \to x_\infty$, then $\phi(x_\infty) 
\le \liminf_{n\to\infty} \phi(x_n)$.
\end{itemize}
Using the standard proof of the Extreme Value Theorem for 
continuous real-valued functions on compact spaces, one can see 
that a lower semi-continuous function achieves a minimum value 
on any compact subset of its domain.

\begin{corollary}\label{sc2zero}
Sequences in $\mathcal{L}(X)$ whose lengths tend to zero have no 
accumulation points in $\lX$, and thus they form closed subsets. 
\end{corollary}

\begin{proof}
Suppose $\{[\gamma_n]\}_{n=1}^\infty$ is a sequence of elements 
of $\lX$ such that $\ell[\gamma_n] \to 0$, and $[\gamma]$ is any 
element of $\mathcal{L}(X)$. Then $[\gamma]$ has a representative 
of length $\eps > 0$, and for any $t \in (0,\eps)$, there exists 
some $N$ such that, for all $n \ge N$, $[\gamma_n]$ does not have 
a representative of length at least $t$, and so 
$[\gamma_n] \notin\tilde{B}(\gamma(t),r)^t$ for $r$ small enough 
and $n\geq N$.
\end{proof}

In particular, any sequence of linear approaches determined by 
saddle connections whose lengths tend to zero form a closed set 
in $\mathcal{L}(X)$.

\begin{corollary}
$\mathcal{L}(X)$ is Hausdorff.
\end{corollary}

\begin{proof}
Suppose $[\gamma_1] \ne [\gamma_2]$. Then there exist 
representatives $\gamma_1$ and $\gamma_2$ of the respective 
classes and some $t$ in their common domain such that 
$\gamma_1(t) \ne \gamma_2(t)$. Set 
$r = \frac{1}{2} d_X(\gamma_1(t),\gamma_2(t))$. Then 
$\tilde{B}(\gamma_1(t),r)^t$ and $\tilde{B}(\gamma_2(t),r)^t$ 
are disjoint open sets containing $[\gamma_1]$ and $[\gamma_2]$, 
respectively.
\end{proof}

\begin{remark}\label{R:notreg}
The space $\mathcal{L}(X)$ is not in general metrizable. Indeed, 
it is not even regular. Recall that a topological space is 
{\em regular} if any point and any closed subset can be separated 
by disjoint open neighborhoods. For example, consider as $X$ the 
geometric construction performed on $\R^2$ as in 
\S\ref{S:definitions}, Example~\ref{Ex:geoconst}. Any neighborhood 
of the horizontal trajectory emanating from the far right of the 
picture intersects any open set containing the sequence of saddle 
connections determined by the segments $I_n$. But, as observed 
following Corollary~\ref{sc2zero}, these saddle connections form 
a closed set.
\end{remark}

\begin{corollary}
$\mathcal{L}(X)$ is second-countable.
\end{corollary}

\begin{proof}
$X$ itself is second-countable because it is a Riemann surface, 
and so it has a countable dense subset $S$. The sets 
$\tilde{B}(x,r)^t$, where $x \in S$, $r \in \Q$, and 
$t \in \Q$ (with $r > 0$, $t > 0$) generate the topology of 
$\mathcal{L}(X)$.
\end{proof}

The following proposition shows that the space $\lX$ is an extension 
of the unit tangent bundle of the flat surface $X$.

\begin{proposition}\label{P:NatEmbedding}
There is a natural topological embedding of the unit tangent bundle 
of $X'$ into $\lX$.
\end{proposition}
\begin{proof}
We denote the unit tangent bundle of $X'$ by $T^1(X')$. Given that 
$X'$ is a translation surface, $T^1(X')=X'\times S^1$. For every 
$(x,\theta)\in T^1(X')\times S^1$ let $i(x,\theta):=[\gamma]\in\lX$ 
be such that $\mathrm{bp}(\gamma)=x$ and 
$\mathrm{dir}(\gamma) = \theta$. Injectivity for $i$ follows from 
the fact that $x$ is a flat point. We now show that $i$ is a 
topological embedding.

Given that $X'\times S^1\hookrightarrow\widehat{X}\times S^1$ is a 
topological embedding and $\mathrm{bp}\times\mathrm{dir}:
\lX\to\widehat{X}\times S^1$ is continuous, the product topology of 
$X'\times S^1$ is contained in the subspace topology of 
$i(X'\times S^1)$. Hence it is sufficient to show that if $x\in X$, 
$t,r>0$ and $[\gamma]\in\tilde{B}(x,r)^t\cap i(T^1(X'))$ are fixed, 
then there exists an open set $U\times V\subset X'\times  S^1$ such 
that:
\[
[\gamma]\in i(U\times V)\subset \tilde{B}(x,r)^t\cap i(T_1(X')).
\]
Since $\lim_{t\to 0}\gamma(t)=x$ is a flat point, we can cover 
$\{x\} \cup \{\gamma(s) \mid s\in(0,t+\delta)\}$, for small 
$\delta>0$, with finitely many $d_x$-balls of fixed radius $r'>0$ 
whose union lies in $X'$. The existence of $U\times V$ follows 
from the fact that $r'>0$ can be chosen arbitrary small.
\end{proof}

\section{Affine maps and rotational components}\label{S:affine}

In the preceding sections we saw how the space $\lX$ generalizes the 
notion of unit tangent bundle. In this section we describe first how 
an affine map between two translation surfaces $X$ and $Y$ induces a 
continuous map from $\lX$ to $\mathcal{L}(Y)$ that generalizes the 
(normalized) derivative. Then we explore some basic facts about 
rotational components. We conclude with a new characterization 
of the pre-compact translation surfaces.

\subsection{Affine maps} Since $\C\cong\R^2$, every translation 
surface $(X,\omega)$ has a canonical \emph{real-affine} translation 
structure.

\begin{definition}
Let $(X,\omega)$ and $(Y,\eta)$ be translation surfaces. An open map 
$f:X\to Y$ is called affine if and only if $f:X'\to Y'$ is affine in 
real-affine charts. That is, in local translation coordinates, 
$f(x,y)=A\cdot(x,y)+(x_0,y_0)$ for some $A \in \mathrm{GL}(2,\R)$. 
\end{definition}

The constant $(x_0,y_0)$ depends on the local coordinates, but the 
differential $A\in\mathrm{GL}(2,\R)$ does not, since all transition 
functions involved are translations. We denote by $\mathrm{Aff}(X,Y)$ 
the set of affine maps from $X$ to $Y$, by $\mathrm{Aff}(X)$ the group 
of affine diffeomorphisms of $X'$ and by $\mathrm{Aff}^+(X)$ the 
subgroup of $\mathrm{Aff}(X)$ formed by orientation preserving maps. 
Remark that every $f\in {\rm Aff}(X,Y)$ has a unique continuous 
extension $\hat{f}:\widehat{X}\to\widehat{Y}$. The action of 
$\mathrm{Aff}(X)$ extends naturally to $\widehat{X}$. Henceforth we 
denote by $\mathrm{Stab}(x)$ the stabilizer of $x\in\widehat{X}$ with 
respect to this action. Recall that the Veech group of $(X,\omega)$ is 
$\Gamma(X):=\{ Df \in \mathrm{GL}(2,\R) \mid f\in\mathrm{Aff}(X)\}$. 
Every countable subgroup of $\mathrm{GL}_+(2,\R)$ without elements of 
norm less than 1 can be realized the Veech group of a translation 
surface having only finite angle singularities, infinite genus and 
one end (see \cite{PSV}).

\subsubsection*{The map $f_*$}
Every $f\in{\rm Aff}(X,Y)$ sends flat points in $X$ to flat points 
in $Y$. For every $\gamma\in\leX$ let $\eps'$ be the total length of 
$f\circ\gamma$. Then there exists a unique unit speed geodesic 
$\eta\in\mathcal{L}^{\eps'}(X)$ parametrizing the image of $f\circ\gamma$ 
and such that $\mathrm{bp}[\eta]=\lim_{t\to 0}f\circ\gamma(t)$. We 
define $f_*:\lX\to\mathcal{L}(Y)$ as $f_*[\gamma]:=[\eta]$.

This definition does not depend on the representative $\gamma$. 
Remark that $(f\circ g)_*=f_*\circ g_*$ and $(\id_X)_*=\id_{\lX}$. 
The following theorem implies that 
$(\mathcal{L}(\,\cdot\,),(\,\cdot\,)_*)$ is a functor from the 
category of translation surfaces satisfying the main hypotheses 
(see \S\ref{S:definitions}) with affine maps to \textbf{Top}, the 
category of topological spaces.

\begin{theorem}
If $f:X\to Y$ is affine, then $f_*:\lX\to\mathcal{L}(Y)$ is 
continuous. 
\end{theorem}

\begin{proof}  
Define $f_\eps:\leX\to\mathcal{L}(Y)$ as 
$f_\eps(\gamma):=f_*[\gamma]$. By the universal property of 
the direct limit, it is sufficient to prove that for every 
$0<\eps'<\eps$ the map $f_\eps$ is continuous and 
$f_\eps=\rho_\eps^{\eps'}\circ f_{\eps'}$, where 
$\rho_\eps^{\eps'}:\leX\to\mathcal{L}^{\eps'}(X)$ denotes the 
restriction map. The equation 
$f_\eps=\rho_\eps^{\eps'}\circ f_{\eps'}$ is clear from the 
definition of $f_*$. Recall that, by Lemma~\ref{L:BxIsBasis} 
in the preceding section, the sets 
$U=\bigcap_{i=1}^n\tilde{B}(y_i,r_i)^{t_i}$ form a basis for the 
topology of $\mathcal{L}(Y)$. Let $\gamma\in f_\eps^{-1}(U)$, 
$f_\eps(\gamma)=[\eta]$ and $s_i\in\R$ be such that 
$f\circ\gamma(s_i)=\eta(t_i)$, $i=1,\ldots,n$.

For every $\gamma\in\leX$ denote by $|f\circ\gamma|$ the total 
length of the image of $f\circ\gamma$. Since $f$ is affine 
(in particular quasiconformal), for every $r>0$ there exists 
$\rho>0$ such that, for every $\gamma_1$ in a $d_\eps$-neighborhood 
of radius $\rho$ around $\gamma$, one has that 
$|f\circ\gamma_1|=K|f\circ\gamma|$, with $K=K(\gamma_1)$ and 
$|K-1|<r$. Furthermore, by the continuity of $f$, such a $\rho>0$ 
can be chosen so that $d_Y(f\circ\gamma_1(s_i),f\circ\gamma(s_i))<r$ 
for every $i=1,\ldots,n$. Let $[\eta_1]:=f_\eps(\gamma_1)$. Remark 
that $f\circ\gamma_1(s_i)=\eta_1(t_i')$, where $t_i'=Kt_i$. Hence, 
if $r>0$ is small enough, then $\eta_1$ is defined at $t=t_i$ and 
$\eta_1(t_i)\in B(y_i,r_i)$ for every $i=1,\ldots,n$.
\end{proof}

\begin{remark}
Every map $f\in\mathrm{Aff}(X,Y)$ acts on the space of directions 
$S^1$ by its normalized differential. In proposition \ref{P:NatEmbedding}, we proved that there is a natural topological 
emmbeding $i:T_1(X')\hookrightarrow \lX$. Every class 
$[\gamma]\in i(T_1(X'))$ is completely determined by the pair 
$(\mathrm{bp}(\gamma),\mathrm{dir}(\gamma))$. By definition the class of 
$f_*[\gamma]$ is determined by $(\mathrm{bp}(f_*[\gamma]),\frac{Df}{|Df|}(\mathrm{dir}(\gamma))$. In other words, the map $f_*$ is the continuous extension 
to $\lX$ of the normalized derivative of the affine map $f$. Remark 
that the  preceding theorem does not follow from the classical 
extension theorems for continuous maps, since $\lX$ is in general 
not regular. On the other hand, remark that 
$\hat{f}\circ\mathrm{bp}[\gamma]=\mathrm{bp}\circ f_*[\gamma]$ 
for every $[\gamma]\in\lX$.
\end{remark}

\begin{corollary}
If $f : X \to Y$ is an affine homeomorphism, then 
$f_*: \lX \to \mathcal{L}(Y)$ is a homeomorphism.
\end{corollary}

\begin{remark}
The converse of this statement does not hold by any means. Even in 
the case of surfaces of finite affine type, if $X$ and $Y$ are in 
the same connected component of a stratum, then $\mathcal{L}(X) $ is homeomorphic to
$\mathcal{L}(Y)$, but this homeomorphism is only induced by 
an affine map $X \to Y$ if they lie in the same 
$\Lie{SL}_2(\R)$-orbit.
\end{remark}

\begin{corollary}
There is a canonical injection from $\mathrm{Aff}(X)$ into 
$\mathrm{Homeo}(\lX)$ and from $\mathrm{Stab}(x)$ into 
$\mathrm{Homeo}(\lx)$ for every $x\in\widehat{X}$.
\end{corollary}

\subsection{Action of $f_*$ on rotational components} 

Recall from the preceding section that the restriction to $f_*$ to 
$T^1(X')$ is given by the \emph{normalized differential} 
$\frac{Df}{|Df|}$. In particular, the action of $f_*$ on a rotational 
component of $\lx$, with $x\in X'$, is given by $\frac{Df}{|Df|}$ as 
well. This situation is not proper to flat points.

\begin{lemma}
Let $x\in\widehat{X}\setminus X'$, $\overline{[\gamma]}$ be a 
rotational component in $\lx$ and $\overline{[\eta]}$ its image 
under $f_{*|}:\lx\to\mathcal{L}(y)$. Denote by 
$A_f:=\frac{Df}{|Df|}$. Then the following diagram commutes:
\begin{equation}
  \label{D:comm}
\xymatrix{
\overline{[\gamma]} \ar[d]^{\mathrm{dir}} \ar[r]^{f_{*|}} &\overline{[\eta]}\ar[d]^{\mathrm{dir}}\\
S^1 \ar[r]^{A_f}          & S^1}
\end{equation}
\end{lemma}

\begin{proof}
Let $[\gamma]\in\overline{[\gamma]}$. If $\eps>0$ is small enough 
there exists a representantive $\gamma:(0,\eps)\to X'$ and flat 
charts $(U,\varphi)$ and $(V,\psi)$ of the image of $\gamma$ and 
$f\circ\gamma$ in $X'$ and $Y'$ respectively, for we are assuming 
that the set of singularities of any translation surface is discrete. 
Without loss of generality we can suppose that 
$0=\lim_{t\to 0}\varphi(\gamma(t))=\lim_{t\to 0}\psi(f(\gamma(t)))$. 
That is, up to composing with a translation, the action in local 
coordinates of $f$ on $\gamma$ is linear. Since the differential of 
$f$ does not depend on local coordinates and we can rescale both 
vectors without changing the direction, we obtain the commutativity 
of \eqref{D:comm}.
\end{proof}

Let $[\gamma_0]\in\overline{[\gamma]}$ and 
$f_{*|}[\gamma_0]=[\eta_0]$. Denote by 
$\alpha_0=\mathrm{dir}[\gamma_0]$, $\beta_0=\mathrm{dir}[\eta_0]$ 
and by $\exp:\R\to S^1$  the universal covering map. Choose 
$t_0\in \exp^{-1}(\alpha_0)$ and $s_0\in \exp^{-1}(\beta_0)$. There 
is a unique lift $\widetilde{A_f}:\R\to\R$ of $A_f:S^1\to S^1$ 
sending $t_0$ to $s_0$. Moreover, there is a unique translation 
embedding $i_0:\overline{[\gamma]}\hookrightarrow\R$ such that 
$i_0([\gamma_0])=t_0$ and making the following diagram commute:
\begin{equation}
  \label{D:comm1}
\xymatrix{
\overline{[\gamma]} \ar[d]^{\mathrm{dir}} \ar@{^{(}->}[r]^{i_0} &\R\ar[d]^{\exp}\\
S^1 \ar[r]^{Id}          & S^1}
\end{equation}
The same is valid for a translation embedding 
$j_0:\overline{[\eta]}\hookrightarrow\R$ satisfying 
$j_0([\eta_0])=s_0$. Hence, if we think of $\R$ as local 
coordinates for the rotational components $\overline{[\gamma]}$ 
and $\overline{[\eta]}$, the action of $f_*$ on a rotational 
component is described globally by the following equation:
\begin{equation}\label{E:describe}
f_{*}\vert_{\overline{[\gamma]}} 
= (j_0^{-1}\circ\widetilde{A_f}\circ i_0)
\end{equation}

\begin{definition}
We call an area-preserving affine automorphism of a flat surface 
parabolic, elliptic, or hyperbolic according to whether the image 
of $Df$ in $\mathrm{PSL}(2,\R)$ is parabolic, elliptic or hyperbolic, 
respectively. 
\end{definition}

\begin{definition}
Let $\overline{[\gamma]}$ be a rotational component with non-empty 
boundary. We call $\alpha\in S^1$ a limit direction of the rotational 
component if it is the limit of the map 
$\mathrm{dir_|}:\overline{[\gamma]}\to S^1$ as one approaches the 
boundary point of $\overline{[\gamma]}$.
\end{definition}

\begin{proposition}
Suppose that the singular locus $\mathrm{Sing}(X)$ of $\widehat{X}$ 
is finite and that there exists a rotational component 
$\overline{[\gamma]}$ of finite length $\lambda$.
\begin{enumerate}
\item If $f\in \mathrm{Aff}(X)$ is parabolic and 
$\lambda \not\equiv 0 \pmod{\pi}$, or 
\item If $f\in {\rm Aff}(X)$ is hyperbolic and the limit directions of $\overline{[\gamma]}$ are not invariant under $A_f$, or 
\item If $f\in {\rm Aff}(X)$ is elliptic but its image in 
$\mathrm{PSL}(2,\R)$ is not conjugated to a torsion element, 
\end{enumerate}
then there exists $x_0\in \mathrm{Sing}(X)$ such that 
$\mathcal{L}(x_0)$ has an infinite number of finite length 
rotational components.  
\end{proposition}

\begin{proof}
We proceed by contradiction. Without loss of generality we can suppose that there is a point $x\in \mathrm{Sing}(X)$ such that the rotational component $\overline{[\gamma]}\in\lx$ is fixed by $f_{*|}$. If $f$ is parabolic,  $\widetilde{A_f}:\R\to\R$ is a map whose fixed points form a lattice of the form $\pi\Z+t$, for some $t\in\R$. In particular, it does not preserve the length of any subinterval $I$ whose endpoints are not in the lattice, which is always the case if $\lambda\not\equiv 0\hspace{1mm}{\rm mod(\pi)}$. If $f$ is hyperbolic, the fixed points of $\widetilde{A_f}$ 
form two lattices $\pi\Z+t$, $\pi\Z+s$ for some real numbers $s\neq t$. In particular it does not preserve the length of any subinterval $I$ whose endpoints are not in the union of this two latices. Such is the case if the limit directions of $\overline{[\gamma]}$ are not invariant under $A_f$. If $f$ is elliptic but its image in ${\rm PSL(2,\R)}$ is not conjugated to a torsion element, then no power of $A_f$ fixes a direction in $S^1$. In particular, it cannot fix the limiting directions of $\overline{[\gamma]}$.
\end{proof}

\subsection{Rotational components} 
In this subsection we state and prove some basic facts about 
rotational components. Then we provide a method to detect when a 
linear approach is in the boundary of the rotational component it 
defines and we discuss transverse measures on subsets of $\lX$. 
Finally, we present a characterization for pre-compact translation 
surfaces in the language developed in this article.
Through the examples in \S\ref{SS:examples}, we showed the existence 
of rotational components isometric to $\R$, open intervals, and 
points. In fact, it is not difficult to combine ideas from these 
examples to realize any connected subset of the real line as a 
rotational component. We can detect when a linear approach is in 
the boundary of the rotational component it defines. For this we 
introduce the continuous function $r:X'\to\R^+\cup\infty$ defined 
by
\begin{equation}\label{E:rfunction}
r(x) := \sup\{r>0\hspace{1mm}|\hspace{1mm}B(x,r)\subset X'\} 
= \mathrm{dist}(x,\mathrm{Sing}(X)).
\end{equation}
That is, $r(x)$ is the largest radius of a disk immersed in $X'$ and 
centered at $x$. 
\begin{lemma}\label{L:nonempty}
For every $x\in\widehat{X}$, there exists a rotational component in 
$\lx$ without empty interior. 
\end{lemma}

\begin{proof}
If $x$ is a flat point, then $\mathcal{L}(x) \cong S^1$ and the 
result is clear. So suppose $x\in\widehat{X}$ is a singular point. 
By our main hypothesis, there exists $r>0$ such that 
$B(x,r)\cap \mathrm{Sing}(X)=\{x\}$. Let $y\in B(x,r)$ be a flat 
point. Then $r(x)<r$ and $x\in\partial B(y,r(x))$. Let 
$\gamma:(0,r(x))\to X'$ be the linear approach contained in the 
segment joining $y$ to $\lim_{t\to 0}\gamma(t)=x$. Then 
$\overline{[\gamma]}$ contains an open interval centered at 
$[\gamma]$ of length at least $\pi$.
\end{proof}

\begin{lemma}\label{L:obstruction}
A class $[\gamma]\in\lX$ is in the boundary of a rotational 
component if and only if it contains a representative $\gamma$ such 
that such that for every $M>0$ there exists 
$t\in(0,\eps)$ for which $r(\gamma(t))<Mt$. 
\end{lemma}

\begin{proof}
We address first sufficiency. Let $\gamma:(0,\eps)\to X'$ be a 
representative for which there is $M>0$ such that for every 
$t\in(0,\eps)$ we have $(r\circ\gamma)(t)\geq Mt$. Then, there 
exists an open subset $D$ of 
$\cup_{t\in(0,\eps/2)}B(\gamma(t),r(\gamma(t)))$ such that: 
(i) there is a flat chart defined on $D$ onto the $xy$-plane for 
which $\gamma(t)$, $t\in(0,\eps/2)$ is the segment $(0,\eps)$ and 
(ii) in this coordinates the segments 
$\{(s,M's)\mid s\in(0,\eps/2),\,0<M'<M/2\}$ define an open 
$U\subset D$ in $\overline{[\gamma]}$ containing $[\gamma]$. 

For necessity remark that if $[\gamma]$ is an interior point of 
$\overline{[\gamma]}$, then there is an angular sector $(I,c,i_c)$, 
with $I$ an open interval, such that $i_c(U(I,c))$ contains a 
representative $\gamma:(0,\eps)\to X'$ of $[\gamma]$. Without loss 
of generality we suppose that such representative corresponds to the 
middle point of $I$. Given that the image of $i_c$ lies within $X'$, 
we can think of $i_c(U(I,c))$ as the angular sector in the 
$xy$-plane defined by:
\[
\{(x,y) \mid 0<x<\eps,\ |y|\leq Mx\}
\]
for some fixed $M>0$. Hence $(r\circ \gamma)(t)\geq Mt$ for all 
$t\in(0,\eps)$.
\end{proof}

\begin{proposition}\label{P:closure}
Each $\mathcal{L}(x)$ is the closure of the interior of its 
rotational components.
\end{proposition}

\begin{proof}
Let $[\gamma] \in \mathcal{L}(x)$. If $[\gamma]$ itself is contained 
in the interior of a rotational component or if it is a boundary 
point for some rotational component with non-empty interior, 
then we are done. Suppose, then, that it is not, and choose a 
representative path $\gamma : (0,\eps) \to X$. Consider a segment 
$\sigma(t)$ of variable length traveling along $\gamma$, for which 
$\gamma(t)$ is its perpendicular bisector at each time $t$. By 
Lemma~\ref{L:obstruction}, $(r \circ \gamma)(t) \to 0$ as $t \to 0$. 
We may assume that the length of $\sigma(t)$ is a monotonic, 
piecewise constant function along $(0,\eps)$ (necessarily having 
countably infinitely many discontinuities) that tends to $0$ at 
$t \to 0$. Let $s(t)$ be the maximal convex function bounded above 
by the length function of $\sigma(t)$; this is a continuous piecewise 
affine function with countably many points of non-differentiability, 
corresponding to a subset of times when $\sigma(t)$ reaches a 
singularity. Then for every $t \in (0,\eps)$, there is at least one 
trajectory $\gamma_t$ having the same direction as $\gamma$ and 
basepoint at the point of $\sigma(t)$ at distance $s(t)$ from 
$\gamma(t)$ along $\sigma(t)$. Each of these $\gamma_t$ lies in the 
interior of a rotational component (by convexity of $s(t)$). Moreover, 
the $[\gamma_t]$ converge to $[\gamma]$. By taking times $t_n \to 0$ 
where $s(t)$ is non-differentiable, we obtain a sequence 
$[\gamma_{t_n}] \to [\gamma]$ of linear approaches having basepoints 
in $\mathrm{sing}(X)$. Because $x$ is isolated from the rest of 
$\mathrm{Sing}(X)$, the result is proved.
\end{proof}

\begin{remark}\label{R:topnearobstr}
Each rotational component is naturally immersed in $\lX$. However, 
its topology as translation 1-manifold can differ from its topology 
as subspace of $\lX$. Consider the geometric construction 
(\S\ref{SS:examples}, example \ref{Ex:geoconst}) performed on the 
real plane. The resulting translation surface presents just one wild 
singularity $x$ and $\lx$ is composed by two rotational components 
each of which is, as translation 1-manifold, isometric to 
$(0,\infty)$. Let $\gamma_1(t)=(1,t)$ and $x=(1,\frac{1}{2})$. It is 
not difficult to see that for each $0<r<\frac{1}{4}$ the open set 
$\widetilde{B}(x,r)^{\frac{1}{2}}$ contains $\overline{[\gamma_1]}$ 
and $\widetilde{B}(x,r)^{\frac{1}{2}}\cap\overline{[\gamma_1]}$ is 
formed by an infinite family of disjoint open intervals 
$\{(a_n,b_n)\}\subset (0.\infty)$. Given that $\overline{[\gamma_1]}$ 
as a $1$-manifold is \emph{locally connected}, the preceding 
discussion implies that the topology of $\overline{[\gamma_1]}$ as 
subspace of $\lX$ is not equivalent to its topology as a $1$-manifold.
\end{remark}

\subsection{Transverse measures}

Along with the angular metric on each rotational component of 
$\mathcal{L}(X)$ and the length metric on subsets of $\mathcal{L}(X)$ 
lying along the same geodesic trajectory, there is an obvious measure 
to consider on subsets of $\mathcal{L}(X)$ having the same direction. 
For each $\theta \in S^1$, let 
$\mathcal{L}_\theta(X) = \mathrm{dir}\inv(\theta)$, and for each 
$x \in \widehat{X}$, let $\mathcal{L}_\theta(x) = 
\mathcal{L}_\theta(X) \cap \mathcal{L}(x)$. We use 
$\mathcal{F}_\theta$ to denote the foliation of $X$ in the direction 
$\theta$. Elements of $\mathcal{L}_\theta(X)$ may now be thought of 
as germs of (oriented) leaves of $\mathcal{F}_\theta$.

Recall that $\mathcal{F}_\theta$ is obtained by integrating the 
kernel field of the one-form $v \mapsto v \cdot v_{\theta^\perp}$, 
where $v_{\theta^\perp}$ is the vector field of unit-length vectors 
whose direction is rotated $\pi/2$ counterclockwise from the 
direction $\theta$; $\mathcal{F}_\theta$ carries the transverse 
measure $\nu_\theta$, which is the absolute value of this one-form. 
If $f : X \to \R_{\ge0}$ is any non-negative, locally bounded, Borel 
measurable function, then $f\nu_\theta$ is a Borel measure that 
can be integrated, at least, over rectifiable curves in $X$. 

Let $B \subset \mathcal{L}_\theta(X)$ be a Borel subset. A 
{\em representative} of $B$ will be a continuous choice $L_B$ of 
representatives for $[\gamma] \in B$, i.e., a continuous section 
$B \to \tilde{\mathcal{L}}(X)$, so that $L_B([\gamma]) \in 
[\gamma]$ varies continuously in length with respect to the topology 
on $\tilde{\mathcal{L}}(X)$. The {\em length} of a representative 
$L_B$ (which may be infinite) is 
\[
\ell(L_B) 
= \sup_{[\gamma]\in B} \left\{\mathrm{length}(L_B([\gamma]))\right\}.
\]
A piecewise $C^1$ curve $\tau : I \to X$ ($I$ may be open or 
closed, bounded or unbounded) is said to be {\em transverse} to a 
representative $L_B$ if it is transverse to each element of 
$L_B$; a collection of at most countably many piecewise $C^1$ 
curves $\{\tau_i\}$ is {\em full} with respect to $L_B$ if each 
$\tau_i$ is transverse to $L_B$ and every element of $L_B$ 
intersects some $\tau_i$. Each representative $L_B$ of $B$ induces 
a characteristic function $\chi_B : X \to \{0,1\}$ whose support is 
the union of the images of elements of $L_B$. (Note that $\chi_B$ 
is not canonical; it depends on a choice of $L_B$.) Now we define 
the measure $\mu_\theta$ of $B$ by 
\begin{multline*}
\mu_\theta(B) = \\
\limsup_{\ell(L_B)\to0} \left[
  \inf \left\{ \sum_i \int_{\tau_i} \chi_B\,\nu_\theta \,\Big\vert\, 
               \text{$\chi_B$ induced by $L_B$, 
                     $\{\tau_i\}$ full with respect to $L_B$} 
       \right\}
  \right].
\end{multline*}
It is fairly obvious that $\mu_\theta$ is a Borel measure on 
$\mathcal{L}_\theta(X)$. Likewise, it restricts to a Borel measure 
on $\mathcal{L}_\theta(x)$ for any $x \in \widehat{X}$. The measures 
$\mu_\theta$ and $\nu_\theta$ are compatible in the following sense.

\begin{proposition}
Let $\sigma$ be a topological segment in $X$, transverse to 
$\mathcal{F}_\theta$, and let $\tilde\sigma = (\sigma,\theta)$ be 
the corresponding subset of $X \times S^1 \subset \mathcal{L}(X)$. 
Then $\tilde\sigma$ is a Borel subset of $\mathcal{L}_\theta(X)$, 
and $\nu_\theta(\sigma) = \mu_\theta(\tilde\sigma)$.
\end{proposition}
\begin{proof}
Clear.
\end{proof}

\begin{remark}
The three types of measures on subsets of $\mathcal{L}(X)$---%
rotational, geodesic, and transverse---are closely analogous to 
the three types of closed one-parameter subgroups of 
$\mathrm{SL}(2,\R)$. Indeed, rotational components are the orbits 
of a canonical ``partial action'' on $\mathcal{L}(X)$ by 
$\widetilde{\rm{SO}}(2)$, the universal cover of $\rm{SO}(2)$. 
The motions along and transversely to geodesic trajectories are 
akin to the geodesic and horocyclic flow on a hyperbolic surface.
\end{remark}

\subsection{Finite type surfaces}

To conclude this section, we provide a new characterization of the ``classical'' translation surfaces, which are included in the next 
definition.
\begin{definition}
A translation surface has {\em finite affine type} if it has 
finite area and the underlying Riemann surface has finite 
analytic type (that is, it is obtained from a compact Riemann 
surface by finitely many punctures). 
\end{definition}

These are often called ``pre-compact'' translation surfaces in the 
literature (see \cite{GJ}); however, as the examples in \S\ref{SS:examples} 
show, a translation surface of infinite genus may have a metric 
completion which is compact. The condition of finite area is 
necessary to rule out abelian differentials which are holomorphic 
on the surface but have poles at the punctures.

\begin{lemma}
Let $x \in \widehat{X}$. Suppose $\ell$ has a positive lower bound 
on $\lx$. Then $x$ is either a cone point or an infinite-angle 
singularity. In particular, if $\lx$ is compact, then $x$ is a 
cone point.
\end{lemma}

\begin{proof}
If $\ell$ has a positive lower bound on $\lx$, then the direction 
map $\lx \to S^1$ is a covering map, from which the result follows.
\end{proof}

\begin{proposition}
$X$ has finite affine type if and only if $\widehat X$ is compact 
and $\lx$ is compact for every $x \in \widehat X$.
\end{proposition}


\begin{proof}
Suppose that $X$ has finite affine type. Then $X$ may be made 
into a compact Riemann surface $\tilde{X}$ by adding finitely 
many points. The translation structure on $X$ is given by an 
abelian differential, which, because it has finite area, extends 
to an abelian differential on $\tilde{X}$. In this way, 
$\widehat{X}$ is canonically homeomorphic to $\tilde{X}$, 
hence compact, and every point of $\widehat{X}$ is either a 
regular point or a cone point, which implies that $\lx \cong S^1$ 
for every $x$.

Conversely, assume that $\widehat{X}$ is compact and $\lx \cong 
S^1$ for every $x \in \widehat{X}$. Because $\ell$ is lower 
semicontinuous, it has a positive lower bound on each $\lx$. This 
implies, from the preceding lemma, that the singularities of 
$\widehat{X}$ are all cone points or marked points. Therefore the 
conformal structure of $X$ extends to $\widehat{X}$, and the cone 
points of $X$ are all finite-order zeroes of an abelian 
differential on $\widehat{X}$, which means $X$ has finite affine 
type.
\end{proof}

\section{Isometries between neighborhoods of singularities}
\label{S:isometries}

Let $X$ and $Y$ be translation surfaces. In this section, we will 
describe a set of necessary and sufficient conditions for 
$x \in \widehat{X}$ and $y \in \widehat{Y}$ to have isometric 
(or more precisely, translation equivalent) neighborhoods. Given 
$\eps > 0$, let $N_\eps(x)$ and $N_\eps(y)$ denote the 
$\eps$-neighborhoods of $x$ and $y$, respectively, and set 
$N'_\eps(x) = N_\eps(x) \setminus\{x\}$ and $N'_\eps(y) = 
N_\eps(y) \setminus\{y\}$. We assume throughout this section that 
any choice of $\eps$ is made so that $N'_\eps(x)$ and $N'_\eps(y)$ 
contain no singularities; this is possible by our standing 
assumption that the singular sets of $X$ and $Y$ are discrete.

Recall that we have defined the direction function $\mathrm{dir}$ 
from both $\mathcal{L}(x)$ and $\mathcal{L}(y)$ to $S^1$, the maximal 
length function $\ell$ from $\mathcal{L}(x)$ and $\mathcal{L}(y)$ to 
$(0,\infty]$, and the transverse measures $\mu_\theta$ on 
$\mathcal{L}_\theta(x)$ and $\mathcal{L}_\theta(y)$ for all 
$\theta \in S^1$. For our present purposes, we must localize the 
notion of maximal length. Given $\eps > 0$, let 
\[
\ell_\eps[\gamma] = \min \{ \eps, \ell[\gamma] \}.
\]
Also let $\sigma_\eps$ be the involution defined on 
$\ell\inv((0,\eps))$ in $\mathcal{L}(x)$ or $\mathcal{L}(y)$ by 
\[
\sigma_\eps([\gamma(t)]) = [\gamma(\ell(\gamma) - t)].
\]
This is the ``pairing'' function on short saddle connections, since 
each saddle connection on from $x$ to itself, for instance, defines 
two elements of $\mathcal{L}(x)$.


\begin{theorem}\label{T:isomexist}
Let $x \in \widehat{X}$ and $y \in \widehat{Y}$. Then the following are 
equivalent:
\begin{enumerate}
\item There exist $\eps > 0$ and a homeomorphism $F : \mathcal{L}(x) 
\to \mathcal{L}(y)$ such that 
\begin{gather*}
\ell_\eps \circ F = \ell_\eps, \qquad
\sigma_\eps \circ F = F \circ \sigma_\eps, \\
\mathrm{dir} \circ F = \mathrm{dir}, \qquad\text{and}\qquad 
\forall\ \theta \in S^1,\ F^* \mu_\theta = \mu_\theta.
\end{gather*}
\item There exist $\eps' > 0$ and a translation equivalence 
$N'_{\eps'}(x) \to N'_{\eps'}(y)$.
\end{enumerate}
\end{theorem}


\begin{remark}
When $x$ and $y$ are flat points or cone points, this theorem reduces 
to the statement that $x$ and $y$ have isometric neighborhoods if and 
only if they have the same total angle.
\end{remark}

The implication ``(2) $\implies$ (1)'' in Theorem~\ref{T:isomexist} 
is obvious by taking $\eps = \eps'$, so assume (1) holds; we will 
show (2) holds with $\eps' = \eps/2$. The proof is by construction. 
Given $z \in N'_{\eps/2}(x)$, set $\delta_z = d_X(x,z)$, and let 
$\gamma_z$ be a shortest trajectory from $x$ to $z$, meaning 
$\gamma_z(\delta_z) = z$. Note that this implies 
$\ell[\gamma_z] > 2\delta_z$. When $F([\gamma]) = [\eta]$, we will 
write $F(\gamma)(t)$ in place of $\eta(t)$ to avoid introducing new 
symbols. In the proof of the following lemma, we also use $[\gamma] + \theta$, for $\theta \in \R$, to mean the linear approach $[\eta]$ in the rotational component $\overline{[\gamma]}$ such that, with respect to the translation structure on $\overline{[\gamma]}$, $[\gamma]$ and $[\eta]$ differ by $\theta$, when such an $[\eta]$ exists.


\begin{lemma}\label{L:short}
With the above assumptions, $F(\gamma_z)$ is a shortest path from $y$ 
to $F(\gamma_z)(\delta_z)$.
\end{lemma}
\begin{proof}
The assumption that $\gamma_z$ is a shortest path from $x$ to $z$ 
implies that it lies in a rotational component of $\mathcal{L}(x)$ 
having length at least $\pi$ (see proof of Lemma~\ref{L:nonempty}), 
and $\ell([\gamma_z]+\theta) \ge 2\delta_z\cos\theta$ for 
$|\theta| < \pi/2$. Because $\eps > 2\delta_z$ and $F$ preserves 
$\ell_\eps$, we have $\ell(F([\gamma_z])+\theta) \ge 
2\delta_z\cos\theta$ for $|\theta| < \pi/2$, which shows that the 
maximal immersed disk centered at $F(\gamma_z)(\delta_z)$ also has 
radius $\delta_z$.
\end{proof}

For each $z \in N'_{\eps/2}(x)$, choose a shortest path $\gamma_z$ 
from $x$ to $z$, and set $w = F(\gamma_z)(\delta_z)$. 


\begin{lemma}\label{L:indep}
The point $w \in N'_{\eps/2}(y)$ is independent of the choice of 
$\gamma_z$.
\end{lemma}
\begin{proof}
Let $B(z,\delta_z)$ be the $\delta_z$-neighborhood of $z$. This neighborhood is 
the image of an immersed Euclidean disk $\tilde{B}(z,\delta_z)$. Let $\gamma_1$ 
and $\gamma_2$ be two choices for $\gamma_z$. Then the segments 
$\gamma_i((0,\delta_z))$ are radii of $B(z,\delta_z)$; let $\eta$ be the saddle 
connection between the corresponding points of $\bdy{\tilde{B}(z,\delta_z)}$. 
Set $w_i = F(\gamma_i)(\delta_z)$; we want to show that $w_1 = w_2$. 
For this it suffices to show that the (immersed) triangle formed by 
$\gamma_1$, $\gamma_2$, and $\eta$ is sent to a congruent triangle. 
This follows from the assumptions $\ell_\eps \circ F = \ell_\eps$, 
$\sigma_\eps \circ F = F \circ \sigma_\eps$, and 
$\mathrm{dir} \circ F = \mathrm{dir}$, together with the 
ASA congruence theorem.
\end{proof}

Lemma~\ref{L:indep} implies that we can define 
$f : N'_{\eps/2}(x) \to N'_{\eps/2}(y)$ unambiguously by 
\[
f(z) = F(\gamma_z)(\delta_z).
\]

\begin{lemma}
The map $f$ is a bijection.
\end{lemma}
\begin{proof}
First observe that, by Lemma~\ref{L:short}, the construction of 
$f$ can be applied in the reverse direction using $F\inv$ to obtain 
an inverse map $f\inv$. Thus every point of $N'_{\eps/2}(y)$ is 
covered via $f$ by a point of $N'_{\eps/2}(x)$, which shows that $f$ 
is surjective. To see that it is injective, note that this is just 
Lemma~\ref{L:indep} applied in the reverse direction, {\em i.e.}, 
it is the observation that $f\inv$ is well-defined.
\end{proof}

\begin{lemma}
The map $f$ is a local isometry.
\end{lemma}
\begin{proof}
Let $z \in N'_{\eps/2}(x)$. We want to show that some embedded disk 
centered at $z$ is carried isometrically into $N'(y)$. The idea of 
our proof is to consider ``polar coordinates'' at $z$ and to show 
that these are preserved by $f$. Let $D_z$ be the largest open, 
embedded disk centered at $z$ and contained in $N'_{\eps/2}(x)$; its 
radius is therefore $\min\{\eps/2 - \delta_z, \mathrm{inj\,rad}(z)\}$, 
where $\mathrm{inj\,rad}(z)$ is the injectivity radius of $X$ at $z$ 
(this is at most $\delta_z$).

Let $\gamma_z$ be a shortest path from $x$ to $z$. Given $z' \in D_z$, 
let $\gamma_{z'}$ be a shortest path from $x$ to $z'$ and let $[z,z']$ 
denote the (unique) shortest segment from $z$ to $z'$. If $\gamma_z$ 
and $\gamma_{z'}$ lie in the same rotational component and are 
rotations of each other by an angle $<\pi/2$, then the segments 
$\gamma_z$, $[z,z']$, and $\gamma_{z'}$ form the sides of a Euclidean 
triangle, which is sent to a congruent triangle by construction 
(and an application of the SAS congruence theorem). Otherwise, we 
construct a path that is a union of saddle connections and is 
carried isometrically to $N'_{\eps/2}(y)$.

Let $B(z,\delta_z)$ and $B(z')$ be the open $\delta_z$ and $\delta_{z'}$ 
neighborhoods of $z$ and $z'$, respectively; use the developing map 
of $X$ to lift these to overlapping disks $\ttil{B}(z)$ and 
$\ttil{B}(z')$ in the plane. By assumption, $\partial\ttil{B}(z)$ and 
$\partial\ttil{B}(z')$ each have points that map to $x$; call these 
$x_1$ and $x_2$. If there is a path in $\ttil{B}(z) \cup \ttil{B}(z')$  
between these points, then its length is less than $\eps$, and we are 
done: a quadrilateral is determined by the lengths of three of its 
sides and the angles between them, and these data are preserved by 
$f$. If no such path exists, then the segment in the plane from 
$x_1$ to $x_2$ passes outside of $\ttil{B}(z) \cup \ttil{B}(z')$. 
Let $\eta$ be the shortest path in $X$ homotopic to the union of 
$\gamma_z$, $[z,z']$, and $\gamma_{z'}$, relative to its endpoints. 
Then the lift of $\eta$ to the plane by the developing map is a 
piecewise linear curve $\tilde\eta$, with possibly infinitely many 
points of non-differentiability, occurring at other points that 
project to $x$ along $\eta$; at each such point, $\eta$ turns 
consistently to either the right or the left. Each of these angles 
is preserved by $f$. Thus we only need to confirm that the lengths 
of the straight segments of $\tilde\eta$ are preserved. Each such 
segment $\sigma$ is a union of saddle connections (with length 
$<\eps$) and points that project to $x$. The lengths, direction, and 
pairing of the saddle connections along $\sigma$ are preserved. Length 
is also preserved along all critical trajectories emanating from $x$ 
in the direction $\theta$ perpendicular to $\sigma$. This set of 
trajectories forms a closed subset of $\mathcal{L}_\theta(x)$, so 
its transverse measure is preserved by $f$. Thus the total length of 
$\sigma$ is preserved. We conclude that $\eta$ is sent by $f$ to 
an isometric path in $N_{\eps/2}(y)$.

Thus we see that $d_X(z,z') = d_Y(f(z),f(z'))$, and the angle 
between $\gamma_z$ and $[z,z']$ equals the angle between 
$F(\gamma_z)$ and $[f(z),f(z')]$. This proves the result.
\end{proof}

\begin{proof}[Proof of Theorem~\ref{T:isomexist}]
A bijection between Riemannian manifolds that is a local isometry 
is also an isometry, and so the result follows immediately from the 
preceding lemmas.
\end{proof}

\section{Final remarks}\label{S:remarks}

In this section we compare the space of directions and the Alexadrov cone of $\widehat{X}$ at a singular point $x$ to $\lx$. Both are metric spaces used to extract information from a neighborhood of a point in a metric space. For a more detailed exposition on these objects we refer the reader to \cite{Bu}.

Recall that the \emph{space of directions} at a point $x\in\widehat{X}$ 
is the metric space consisting of curves emanating from $x$ for which a 
comparison angle exists. The corresponding metric is the \emph{upper 
angle} metric $\measuredangle_U(\cdot,\cdot)$. 

Let $\widehat{X}$ be the metric completion of a finite cyclic covering 
of $\C^*$ and $x_0=\mathrm{Sing}(X)$. We index the sheets of this 
covering by $\Z/n\Z$. Let $\gamma_1$ and $\gamma_2$ be two linear 
approaches to $x_0$ whose images do not lie in the same sheet of the 
covering. That is, their distance in the rotational component forming 
$\mathcal{L}(x_0)$ is greater than $2\pi$. A straightforward 
calculation shows that  $\measuredangle_U(\gamma_1,\gamma_2)=\pi$. 
Hence with the space of directions we obtain less information about 
the set of geodesics emanating from $x_0$ than with $\mathcal{L}(x_0)$.

Let $\Gamma_x$ denote the set of geodesics emerging from a singularity 
$x\in\widehat{X}$. Define on  $\Gamma_x\times[0,\infty)$ the 
pseudo-metric:
\begin{equation}\label{pmetric}
d\big((\gamma_1,s_1),(\gamma_2,s_2)\big) := 
\limsup_{t\to 0^+}
\frac{d_X(\gamma_1(ts_1),\gamma_2(ts_2))}{t}
\end{equation}
Let $C_x$ denote the metric space obtained after taking the quotient 
by (\ref{pmetric}) and $\widehat{C_x}$ the corresponding metric 
completion. This metric space is called the \emph{Alexandrov cone} at 
$x$. Suppose that $x$ is an infinite angle singularity. It is not hard 
to find a pair of linear approaches $[\gamma_1]\neq[\gamma_2]$ to $x$ 
and a pair of positive real numbers $s_1$, $s_2$ such that the 
distance \eqref{pmetric} between the points $(\gamma_1,s_1)$ and 
$(\gamma_2,s_2)$ is arbitrarily small and at the same time the 
distance between $[\gamma_1]$ and $[\gamma_2]$ in the corresponding 
double spire is arbitrarly large. In other words, from the metric point 
of view the Alexandrov cone cannot tell apart linear approaches in 
the same rotational component that are far away from each other.

More seriously, because the space of directions and the Alexandrov 
cone are metric spaces, they lose certain convergence information 
contained in $\mathcal{L}(x)$; see remark~\ref{R:notreg}. However, 
the Alexandrov cone can be completely recovered from $\mathcal{L}(x)$ 
by separating into rotational components.

\begin{remark}
For ``good'' metric spaces, the Gromov--Hausdorff tangent cone at a 
point $x_0$ is nothing but the metric cone over the space of directions 
at $x_0$, where as for ``bad'' (\emph{e.g.}, non locally compact, 
non boundedly compact) this cone might not be well defined (see 
\cite[\S8.2]{Bu} for more details). Hence, with the Gromov--Hausdorff 
cone at a singularity we obtain (if any) less information about the 
set of geodesics emanating from $x_0$ than with $\mathcal{L}(x_0)$.
\end{remark}

Most of our constructions only rely on the fact that a translation 
surface is a Riemannian manifold; indeed, in many cases only an affine 
connection is required. It would be interesting to know if these 
constructions have applications in other areas.



\end{document}